\documentclass[11pt]{amsart}

\usepackage{amssymb,amsmath,amsthm}
\usepackage{graphicx}
\usepackage{color}
\usepackage{thmtools}
\usepackage{thm-restate}
\usepackage[shortlabels]{enumitem}
\usepackage{mathrsfs}
\usepackage[utf8]{inputenc}
\usepackage[T1]{fontenc}
\usepackage{dsfont}
\usepackage{soul}
\usepackage[left=1in,right=1in,top=1in,bottom=1in]{geometry}

\calclayout

\newtheorem{theorem}{Theorem}[section]
\newtheorem*{theorem*}{Theorem}
\newtheorem*{claim*}{Claim}

\newtheorem{proposition}[theorem]{Proposition}

\newtheorem{question}[theorem]{Question}
\newtheorem{corollary}[theorem]{Corollary}
\newtheorem{fact}[theorem]{Fact}

\theoremstyle{definition}

\theoremstyle{remark}
\newtheorem{remark}[theorem]{Remark}

\numberwithin{equation}{section}

\newcommand{\eps}{\varepsilon}

\newcommand{\N}{\mathbb{N}}
\newcommand{\R}{\mathbb{R}}

\newcommand{\ol}{\overline}

\newcommand{\rstr}{\restriction}

\newcommand{\sm}{\setminus}

\DeclareMathOperator{\supp}{supp}

\newcommand{\seq}[2]{\big\langle#1\colon\ #2\big\rangle}
\newcommand{\seqn}[1]{\big\langle#1\colon\ n\io\big\rangle}

\newcommand{\finsub}[1]{\left[#1\right]^{<\omega}}

\newcommand{\io}{\in\mathbb{N}}

\newcommand{\fso}{\finsub{\omega}}

\begin{document}

\title[Complementability of $c_0$ in spaces $C(K\times L)$]{On complementability of $c_0$ in spaces $C(K\times L)$}
\author[J. K\k{a}kol]{Jerzy K\k{a}kol}
\address{Faculty of Mathematics and Computer Science, A. Mickiewicz University, Pozna\'n, Poland, and Institute of Mathematics, Czech Academy of Sciences, Prague, Czech Republic.}
\email{kakol@amu.edu.pl}
\author[D.\ Sobota]{Damian Sobota}
\address{Universit\"at Wien, Institut f\"ur Mathematik, Kurt G\"odel Research Center, Wien, Austria.}
\email{ein.damian.sobota@gmail.com}
\urladdr{www.logic.univie.ac.at/~{}dsobota}
\author[L. Zdomskyy]{Lyubomyr Zdomskyy}
\address{Universit\"at Wien, Institut f\"ur Mathematik, Kurt G\"odel Research Center, Wien, Austria.}
\email{lzdomsky@gmail.com}
\urladdr{www.logic.univie.ac.at/~{}lzdomsky}
\thanks{The research of the first  named author is supported by the GA\v{C}R project 20-22230L and RVO: 67985840. The second and third named authors were supported by the Austrian Science Fund FWF, Grants I 2374-N35, I 3709-N35, M 2500-N35, I 4570-N35.}


\begin{abstract}
Using elementary probabilistic methods, in particular a variant of the Weak Law of Large Numbers related to the Bernoulli distribution, we prove that for every infinite compact spaces $K$ and $L$ the product $K\times L$ admits a sequence $\langle\mu_n\colon n\in\mathbb{N}\rangle$ of normalized signed measures with finite supports which converges to $0$ with respect to the weak* topology of the dual Banach space $C(K\times L)^*$. Our approach is completely constructive---the measures $\mu_n$ are defined by an explicit simple formula. We also show that this result generalizes the classical theorem of Cembranos and Freniche which states that for every infinite compact spaces $K$ and $L$ the Banach space $C(K\times L)$ contains a complemented copy of the space $c_0$.
\end{abstract}

\maketitle

\section{Introduction}

As usual, for a compact (Hausdorff) space $X$ we denote by $C(X)$ the Banach space of continuous real-valued functions on $X$ and by $c_0$ the Banach space of all real-valued sequences which converge to $0$, both endowed with the supremum norm.

It is an easy observation that for every infinite compact space $X$ the space $C(X)$ contains a closed linear subspace isomorphic to the space $c_0$. This subspace need not be complemented---e.g., for $X=\beta\N$, the \v{C}ech--Stone compactification of the set $\N$ of non-negative integers equipped with the discrete topology, the space $C(X)$ does not contain any complemented copies of $c_0$. However, if $X$ is a product of two infinite compact spaces, then $C(X)$ always contains a complemented copy of $c_0$, as was proved by Cembranos \cite{Cem84} and Freniche \cite{Fre84}.

\begin{theorem}[Cembranos--Freniche]\label{theorem:cemfre}
For every infinite compact spaces $K$ and $L$ the Banach space $C(K\times L)$ contains a complemented copy of the space $c_0$.
\end{theorem}

A crucial step in both of the proofs is an application of the classical Josefson--Nissenzweig theorem which asserts that for every infinite-dimensional Banach space $E$ there exists a sequence $\seqn{x_n^*}$ in the dual space $E^*$ such that $\|x_n^*\|=1$ for every $n\io$ and $x_n^*(x)\to0$ for every $x\in E$ (that is, $\seqn{x_n^*}$ is convergent to $0$ with respect to the weak* topology). Since all standard proofs of the Josefson--Nissenzweig theorem are rather intricate and non-constructive, it is hard to deduce from the proofs of Cembranos and Freniche how the constructed complemented copy of $c_0$ in a given space $C(K\times L)$ basically looks like or what properties it has. In this paper we address this issue and provide an elementary and constructive proof of the following strengthening of Theorem \ref{theorem:cemfre}.

\begin{theorem}\label{theorem:main}
For every infinite compact spaces $K$ and $L$ there is a sequence $\seqn{\mu_n}$ of normalized signed measures on $K\times L$ with finite supports which is weak* convergent to $0$.
\end{theorem}

Note that, by the virtue of the classical Riesz representation theorem, for every infinite compact space $X$, the Josefson--Nissenzweig theorem can be expressed in the following way: there is a sequence $\seqn{\mu_n}$ of normalized signed regular Borel measures on $X$ which converges to $0$ with respect to the weak* topology of the dual space $C(X)^*$. Hence Theorem \ref{theorem:main} might be treated in the first place as a special case of the Josefson--Nissenzweig theorem. However, as said above, in contrast to the latter general result, our proof of Theorem \ref{theorem:main} is completely constructive---for every compact spaces $K$ and $L$ the measures $\mu_n$ are given by an explicit simple formula. To prove that they weak* converge to $0$ we use elementary tools from probability theory, in particular a variant of the Weak Law of Large Numbers related to the Bernoulli distribution.

The Cembranos--Freniche theorem can be deduced from Theorem \ref{theorem:main} in the following way, using some $C_p$-theory. Recall that if $X$ is a Tychonoff space, then $C_p(X)$ denotes the space of all continuous real-valued functions on $X$ equipped with the pointwise topology. Similarly, by $(c_0)_p$ we mean the space $c_0$ but endowed with the topology inherited from the product $\R^\N$. The next corollary is an immediate consequence of Theorem \ref{theorem:main} and the following result due to Banakh, K\k{a}kol, and \'{S}liwa \cite[Theorem 1]{BKS19}: a Tychonoff space $X$ admits a sequence $\seqn{\mu_n}$ of finitely supported signed measures such that $\big\|\mu_n\big\|=1$ for every $n\io$ and $\int_Xfd\mu_n\to0$ for every $f\in C_p(X)$ if and only if the space $C_p(X)$ contains a complemented copy of the space $(c_0)_p$. Of course, if $X$ is compact, then the condition that $\int_Xfd\mu_n\to0$ for every $f\in C_p(X)$ is equivalent to $\seqn{\mu_n}$ being convergent to $0$ with respect to the weak* topology of $C(X)^*$.

\begin{corollary}\label{cor:main}
For every infinite compact spaces $K$ and $L$ the space $C_p(K\times L)$ contains a complemented copy of the space $(c_0)_p$.
\end{corollary}

\noindent With an aid of the Closed Graph Theorem, the result of Cembranos and Freniche can be easily deduced from the above corollary, see Proposition \ref{prop:c0p_c0} for more details.

Theorem \ref{theorem:main} gives also applications to Grothendieck $C(X)$-spaces. Recall that a Banach space $E$ is \textit{Grothendieck} if every sequence $\seqn{x_n^*}$ in the dual space $E^*$ which is weak* convergent to $0$, converges also weakly. It was proved by Cembranos \cite[Corollary 2]{Cem84} (cf. also \cite[Proposition 5.3]{Sch82}) that, for every infinite compact space $X$, the Banach space $C(X)$ is Grothendieck if and only if $C(X)$ does not contain any complemented copy of $c_0$. The latter result and Theorem \ref{theorem:cemfre} immediately yield the following theorem of Khurana \cite[Theorem 2]{Khu78}: for every infinite compact spaces $K$ and $L$ the space $C(K\times L)$ is never Grothendieck, that is, there exists a sequence $\seqn{\mu_n}$ of normalized signed regular Borel measures on $K\times L$ which converges weak* to $0$ but not weakly. Theorem \ref{theorem:main} strengthens this observation---there must even exist such $\seqn{\mu_n}$ which consists only of finitely supported measures (note that, by the Schur property, such $\seqn{\mu_n}$ cannot be weakly convergent to $0$). This yields indeed a strengthening of Khurana's theorem as there exists a compact space $X$ such that $C(X)$ is not Grothendieck but for every separable closed subspace $Y$ of $X$ the space $C(Y)$ is Grothendieck (see e.g. \cite[Example 10]{KM20}).

Let us finish by mentioning the following problem. 
It is well-known that the space $C(\beta\N)$ is \textit{prime}, that is, every complemented infinite-dimensional subspace of $C(\beta\N)$ is isomorphic to $C(\beta\N)$ (see \cite{Lin67}). On the other hand, $C(\beta\N\times\beta\N)$ is obviously not prime, we do not know however whether there exist complemented subspaces different than the ones listed below (cf. \cite{Can21}).

\begin{question}
Does there exist a complemented infinite-dimensional subspace of $C(\beta\N\times\beta\N)$ which is not isomorphic to any of the following spaces: $C(\beta\N\times\beta\N)$, $C(\beta\N)$, $c_0$, $c_0\oplus C(\beta\N)$, and $c_0(C(\beta\N))$?
\end{question}

\section{Proof of Theorem \ref{theorem:main}}




Let us start with some notation and terminology. All topological spaces considered by us are assumed to be Tychonoff. For a Tychonoff space $X$, by $Bor(X)$ we denote the Borel $\sigma$-field of $X$. A function $\mu\colon Bor(X)\to\R$ is \textit{a measure} on $X$, if $\mu$ is additive, inner regular with respect to closed sets and outer regular with respect to open sets, and has bounded variation (that is, $\|\mu\|=\sup\big\{|\mu(A)|+|\mu(B)|\colon A,B\in Bor(X), A\cap B=\emptyset\big\}<\infty$). If $\mu$ attains negative values, then we say that $\mu$ is \textit{signed}. If $\|\mu\|=1$, then we say that $\mu$ is \textit{normalized}. \textit{The support} of $\mu$ is denoted by $\supp(\mu)$. If there are finite sequences $x_1,\ldots,x_n\in X$ (with all points pairwise distinct) and $\alpha_1,\ldots,\alpha_n\in\R$ such that $\mu=\sum_{i=1}^n\alpha_i\delta_{x_i}$, where by $\delta_x$ we denote the one-point measure at a point $x\in X$, then we say that $\mu$ is \textit{finitely supported}. Note that in this case we have $\supp(\mu)=\big\{x_1,\ldots,x_n\big\}$ and $\|\mu\|=\sum_{i=1}^n\big|\alpha_i\big|$. We assume that $0\in\N$ and by $\N_+$ we denote the set of positive integers, i.e. $\N_+=\N\sm\{0\}=\{1,2,3,\ldots\}$. We identify every $n\in\N$ with the set $\{0,\ldots,n-1\}$. If $A$ is a set, then by $|A|$ we denote its cardinality.

\medskip

For every $n\io_+$
put $\Omega_n=\{-1,1\}^n$ and $\Sigma_n=n\times\{n\}$ (so
$\big|\Omega_n\big|=2^n$ and $\big|\Sigma_n\big|=n$). To simplify
the notation, we will usually write $i\in\Sigma_n$ instead of
$(i,n)\in\Sigma_n$---this should cause no confusion. Put also
$\Omega=\bigcup_{n\io_+}\Omega_n$ and
$\Sigma=\bigcup_{n\io_+}\Sigma_n$, and endow these two sets with the
discrete topology. This way, we can think of the product space
$\Omega\times\Sigma$ as a countable union of pairwise disjoint
discrete rectangles $\Omega_k\times\Sigma_m$ of size $m2^k$---the
rectangles $\Omega_n\times\Sigma_n$, lying along the diagonal, will
bear a special meaning, namely, they will be the supports of
measures from the special sequence $\seq{\mu_n}{n\io_+}$ on the space
$\beta\Omega\times\beta\Sigma$ (that is, on the product of the \v{C}ech--Stone compactifications of $\Omega$ and $\Sigma$) defined as follows ($n\io_+$):
\[\mu_n=\sum_{\substack{s\in\Omega_n\\i\in\Sigma_n}}\frac{s(i)}{n2^n}\delta_{(s,i)}.\]
Then, $\supp\big(\mu_n\big)=\Omega_n\times\Sigma_n$, so $\big|\supp\big(\mu_n\big)\big|=n2^n$,
$\big\|\mu_n\big\|=1$, and
\[\pi_i\big[\supp\big(\mu_n\big)\big]\cap\pi_i\big[\supp\big(\mu_{n'}\big)\big]=\emptyset\]
for every $n\neq n'$ and $i\in\{0,1\}$ (here $\pi_i$ denotes the projection on the $i$-th coordinate). Note that for each $n\io_+$ and any two sets $A\in\wp(\Omega)$ and $B\in\wp(\Sigma)$ we have:
\[\tag{$\dagger$}\big|\mu_n\big([A]\times[B]\big)\big|\le\frac{\big|A\cap \Omega_n\big|}{2^n}\cdot\frac{\big|B\cap \Sigma_n\big|}{n},\]
where $[A]$ and $[B]$ always denote the clopen subsets of $\beta\Omega$ and $\beta\Sigma$ corresponding in the sense of the Stone duality to $A$ and $B$, respectively---since $\beta\Omega$ and $\beta\Sigma$ are extremely disconnected, we have $[A]=\ol{A}^{\beta\Omega}$ and $[B]=\ol{B}^{\beta\Sigma}$.

\medskip

In the next proposition we will prove that the sequence $\seqn{\mu_n}$, as defined above, is weak* convergent to $0$ on the space $\beta\Omega\times\beta\Sigma$. However, before we do that, we need to provide a bit of explanation of probability tools we use in the proof. For every $n\io_+$ and $i\in n$ define the function $X_i\colon\Omega_n\to\{0,1\}$ as follows: $X_i(r)=1$ if and only if $r(i)=1$, where $r\in\Omega_n$. Put $S_n=\sum_{i=0}^{n-1}X_i$, so $S_n\colon\Omega_n\to n$ is the function computing the number of $1$'s in the argument sequence $r\in\Omega_n$. For a finite set $A\in\fso$, let $P_A$ denotes the standard product probability on $\{-1,1\}^A$ (assigning $1/2^{|A|}$ to each elementary event, i.e. $P_A(\{r\})=1/2^{|A|}$ for each $r\in\{-1,1\}^{|A|}$). Recall that for every $k\le n$ it holds:
\[P_n(S_n=k)=P_n\big(\big\{r\in\Omega_n\colon\ S_n(r)=k\big\}\big)={n\choose k}1/2^n.\]
We will need the following fact, being a variant of the Weak Law of Large Numbers, which estimates the probability that $S_n(r)$ has value ``far'' (with respect to $\eps$) from $n/2$, i.e. that ``$r$ contains \emph{significantly} more (with respect to $\eps$) $1$'s than $-1$'s, or \textit{vice versa}''.

\begin{fact}\label{fact:bollobas}
If $n\io_+$ and $\eps\in\big(0,1/12\big]$ are such numbers that $n\ge48/\eps$, then:
\[P_n\big(\big|S_n-n/2\big|\ge\eps n/2\big)\le
\frac{\sqrt{2}}{\eps\sqrt{n}}.\]
\end{fact}
\begin{proof}
See Bollob\'as \cite[Theorem 1.7.(i)]{Bol01}.
\end{proof}

\medskip

We are ready to prove the aforementioned auxiliary proposition.

\begin{proposition}\label{prop:fsjnp_product_omega_sigma}
The sequence $\seq{\mu_n}{n\io_+}$ defined above is convergent to $0$ with respect to the weak* topology of the dual space $C(\beta\Omega\times\beta\Sigma)^*$.
\end{proposition}
\begin{proof}
Since $\beta\Omega\times\beta\Sigma$ is a totally disconnected compact space, to prove that $\seq{\mu_n}{n\io_+}$ is weak* convergent to $0$ it is enough to show that it converges to $0$ on every clopen subset of the form $[A]\times[B]$, where $A\in\wp(\Omega)$ and $B\in\wp(\Sigma)$ (to see this, use e.g. the Stone--Weierstrass theorem applied to the linear subspace of simple functions defined on clopen subsets of $\beta\Omega\times\beta\Sigma$). So let us fix two such sets $A$ and $B$.

Fix $\eps\in\big(0,1/12\big]$ and put:
\[I_0=\big\{n\io_+\colon\ \big|B\cap\Sigma_n\big|<2/\eps^4\big\}\]
and
\[I_1=\N_+\sm I_0=\big\{n\io_+\colon \big|B\cap\Sigma_n\big|\ge2/\eps^4\big\}.\]
For each $i\in\{0,1\}$ we will find $N_i\io_+$ such that for every $n\ge N_i$, $n\in I_i$, it holds
\[\big|\mu_n\big([A]\times[B]\big)\big|=\big|\mu_n\big(\big[A\cap\Omega_n\big]\times\big[B\cap\Sigma_n\big]\big)\big|<2\eps.\]

\medskip

We first look for $N_0$. By ($\dagger$) for every $n\in I_0$ we have:
\[\big|\mu_n\big([A]\times[B]\big)\big|\le\frac{\big|A\cap \Omega_n\big|}{2^n}\cdot\frac{\big|B\cap\Sigma_n\big|}{n}\le\frac{\big|B\cap\Sigma_n\big|}{n}<\frac{2}{n\eps^4},\]
so if $I_0$ is infinite, then there exists $N_0\io_+$ such that for every $n\ge N_0$, $n\in I_0$, we have:
\[\big|\mu_n\big([A]\times[B]\big)\big|<2\eps.\]
If, on the other hand, $I_0$ is finite, then simply set $N_0=1+\max I_0$.

\medskip

Let us now look for $N_1$. We assume that $I_1$ is non-empty---otherwise simply set $N_1=0$.  For every $n\in I_1$ define the set $\Delta_{n,\eps}$ as follows:
\[\Delta_{n,\eps}=\Big\{s\in\Omega_n\colon\ \Big|\big|\big\{i\in B\cap\Sigma_n\colon\ s(i)=1\big\}\big|-\frac{\big|B\cap\Sigma_n\big|}{2}\Big|\ge\eps\frac{\big|B\cap\Sigma_n\big|}{2}\Big\},\]
so $\Delta_{n,\eps}$ denotes the event that $s\in\Omega_n$ is ``far'' (with respect to $\eps$) from having the same numbers of $1$'s and $-1$'s when restricted to the set $B$. If we put similarly:
\[\Gamma_{n,\eps}=\Big\{s\in\{-1,1\}^{B\cap\Sigma_n}\colon\ \Big|\big|\big\{i\in B\cap\Sigma_n\colon\ s(i)=1\big\}\big|-\frac{\big|B\cap\Sigma_n\big|}{2}\Big|\ge\eps\frac{\big|B\cap\Sigma_n\big|}{2}\Big\},\]
then we trivially have:
\[\tag{$\times$}\Delta_{n,\eps}=\Gamma_{n,\eps}\times\{-1,1\}^{\Sigma_n\sm B}.\]
Using this, for every $n\in I_1$ we will estimate the following values (see (1) and (3)):
\[\tag{$*$}\big|\mu_n\big(\big[A\cap\Delta_{n,\eps}\big]\times[B\cap\Sigma_n]\big)\big|=\Big|\sum_{\substack{s\in A\cap\Delta_{n,\eps}\\i\in B\cap\Sigma_n}}\frac{s(i)}{n2^n}\Big|\]
and
\[\tag{$**$}\big|\mu_n\big(\big[A\cap\big(\Omega_n\sm\Delta_{n,\eps}\big)\big]\times[B\cap\Sigma_n]\big)\big|=\Big|\sum_{\substack{s\in A\cap(\Omega_n\sm\Delta_{n,\eps})\\i\in B\cap\Sigma_n}}\frac{s(i)}{n2^n}\Big|.\]
Note that:
\[\big|\mu_n\big([A]\times[B]\big)\big|\le\big|\mu_n\big(\big[A\cap\Delta_{n,\eps}\big]\times[B\cap\Sigma_n]\big)\big|+\big|\mu_n\big(\big[A\cap\big(\Omega_n\sm\Delta_{n,\eps}\big)\big]\times[B\cap\Sigma_n]\big)\big|,\]
so obtaining ``good'' estimations of ($*$) and ($**$) will finish the proof.

\medskip

Fix $n\in I_1$ and let us start with the estimation of ($*$). Note that $\big|B\cap\Sigma_n\big|\ge48/\eps$, so recall ($\times$) and apply Fact \ref{fact:bollobas} with the set $B\cap\Sigma_n$ instead of the set $n=\{0,\ldots,n-1\}$ to get that: 
\[\tag{$0$}P_n\big(\Delta_{n,\eps}\big)=P_{B\cap\Sigma_n}\big(\Gamma_{n,\eps}\big)\cdot P_{\Sigma_n\sm B}\Big(\{-1,1\}^{\Sigma_n\sm B}\Big)=P_{B\cap\Sigma_n}\big(\Gamma_{n,\eps}\big)\le\frac{\sqrt{2}}{\eps\sqrt{\big|B\cap\Sigma_n\big|}}.\]
It holds that:
\[\Big|\sum_{\substack{s\in A\cap\Delta_{n,\eps}\\i\in B\cap\Sigma_n}}\frac{s(i)}{n2^n}\Big|\le\sum_{\substack{s\in A\cap\Delta_{n,\eps}\\i\in B\cap\Sigma_n}}\frac{1}{n2^n}=\frac{\big|A\cap\Delta_{n,\eps}\big|\cdot\big|B\cap\Sigma_n\big|}{n2^n}\le\]
\[\le\frac{\big|\Delta_{n,\eps}\big|\cdot n}{n2^n}=P_n\big(\Delta_{n,\eps}\big),\]
so by (0) we get the following estimation of ($*$):
\[\tag{$1$}\Big|\sum_{\substack{s\in A\cap\Delta_{n,\eps}\\i\in B\cap\Sigma_n}}\frac{s(i)}{n2^n}\Big|\le\frac{\sqrt{2}}{\eps\sqrt{\big|B\cap\Sigma_n\big|}}.\]

\medskip

We now estimate ($**$). For every $s\in\Omega_n\sm \Delta_{n,\eps}$ we have:
\[\big|\sum_{i\in B\cap\Sigma_n}s(i)\big|=\Big|\big|\big\{i\in B\cap\Sigma_n\colon\ s(i)=1\big\}\big|-\big|\big\{i\in B\cap\Sigma_n\colon\ s(i)=-1\big\}\big|\Big|\le\]
\[\le\Big|\big|\big\{i\in B\cap\Sigma_n\colon\ s(i)=1\big\}\big|-\frac{\big|B\cap\Sigma_n\big|}{2}\Big|+\Big|\big|\big\{i\in B\cap\Sigma_n\colon\ s(i)=-1\big\}\big|-\frac{\big|B\cap\Sigma_n\big|}{2}\Big|<\]
\[<2\cdot\eps\frac{\big|B\cap\Sigma_n\big|}{2}=\eps\big|B\cap\Sigma_n\big|,\]
so:
\[\tag{$2$}\big|\sum_{i\in B\cap\Sigma_n}s(i)\big|<\eps\big|B\cap\Sigma_n\big|.\]
Next, it holds:
\[\Big|\sum_{\substack{s\in A\cap(\Omega_n\sm\Delta_{n,\eps})\\i\in B\cap\Sigma_n}}\frac{s(i)}{n2^n}\Big|\le\frac{1}{n2^n}\Big|\sum_{\substack{s\in A\cap(\Omega_n\sm\Delta_{n,\eps})\\i\in B\cap\Sigma_n}}s(i)\Big|=\]
\[=\frac{1}{n2^n}\Big|\sum_{s\in A\cap(\Omega_n\sm\Delta_{n,\eps})}\ \sum_{i\in B\cap\Sigma_n}s(i)\Big|\le\frac{1}{n2^n}\sum_{s\in A\cap(\Omega_n\sm\Delta_{n,\eps})}\big|\sum_{i\in B\cap\Sigma_n}s(i)\big|\le\]
\[\le\frac{1}{n}\max\Big\{\big|\sum_{i\in B\cap\Sigma_n}s(i)\big|\colon\ s\in\Omega_n\sm\Delta_{n,\eps}\Big\},\]
so by (2):
\[\tag{$3$}\Big|\sum_{\substack{s\in A\cap(\Omega_n\sm\Delta_{n,\eps})\\i\in B\cap\Sigma_n}}\frac{s(i)}{n2^n}\Big|<\frac{\eps\big|B\cap\Sigma_n\big|}{n}\le\frac{\eps n}{n}=\eps.\]

\medskip

Using (1), (3), and the fact that $\big|B\cap\Sigma_n\big|\ge2/\eps^4$, we conclude that for every $n\in I_1$ we have:
\[\big|\mu_n\big([A]\times[B]\big)\big|=\Big|\sum_{\substack{s\in A\cap\Omega_n\\i\in B\cap\Sigma_n}}\frac{s(i)}{n2^n}\Big|\le\Big|\sum_{\substack{s\in A\cap\Delta_{n,\eps}\\i\in B\cap\Sigma_n}}\frac{s(i)}{n2^n}\Big|+\Big|\sum_{\substack{s\in A\cap(\Omega_n\sm\Delta_{n,\eps})\\i\in B\cap\Sigma_n}}\frac{s(i)}{n2^n}\Big|<\]
\[<\frac{\sqrt{2}}{\eps\sqrt{\big|B\cap\Sigma_n\big|}}+\eps\le\eps+\eps=2\eps.\]
It follows that if for $N_1$ we take any number from $I_1$, e.g. we set $N_1=\min I_1$, then for every $n\in I_1$, $n\ge N_1$, we have:
\[\big|\mu_n\big([A]\times[B]\big)\big|<2\eps.\]

We finish the proof by denoting $N=\max\big(N_0,N_1\big)$ and seeing that for every $n\ge N$ we obviously have the same inequality, i.e.:
\[\big|\mu_n\big([A]\times[B]\big)\big|<2\eps.\]
Since $\eps\in\big(0,1/12\big]$ is arbitrary, it holds that $\lim_{n\to\infty}\mu_n\big([A]\times[B]\big)=0$ and hence $\seq{\mu_n}{n\io_+}$ is weak* convergent to $0$.
\end{proof}


\begin{remark}
Let us note that from the above proof we can deduce for every $n\io_+$ the following estimations of the value $\big|\mu_n\big([A]\times[B]\big)\big|$, depending only on the size of the intersection $B\cap\Sigma_n$:
\[
\big|\mu_n\big([A]\times[B]\big)\big|\le
\begin{cases}
\frac{|B\cap\Sigma_n|}{n}&,\text{ if }\big|B\cap\Sigma_n\big|<\frac{2}{\eps^4},\\
\frac{\sqrt{2}}{\eps\sqrt{|B\cap\Sigma_n|}}+\frac{\eps|B\cap\Sigma_n|}{n}&,\text{ if }|B\cap\Sigma_n|\ge\frac{2}{\eps^4}.
\end{cases}
\]
\end{remark}

\medskip

We are in the position to prove the main result of this paper, Theorem \ref{theorem:main}.

\begin{proof}[Proof of Theorem \ref{theorem:main}]
First, notice that the space $\N$ of all non-negative integers, endowed with the discrete topology, is homeomorphic to both $\Omega$ and $\Sigma$, so the \v{C}ech--Stone compactifications $\beta\N$, $\beta\Omega$, and $\beta\Sigma$ are mutually homeomorphic. Consequently, by Proposition \ref{prop:fsjnp_product_omega_sigma}, $\beta\N\times\beta\N$ admits a weak* convergent to $0$ sequence $\seqn{\nu_n}$ of finitely supported normalized signed measures with pairwise disjoint supports, contained completely in $\N\times\N$ (as measures $\mu_n$'s defined above on $\beta\Omega\times\beta\Sigma$ have pairwise disjoint supports contained in $\Omega\times\Sigma$).

Let $D$ and $E$ be discrete countable subsets of $K$ and $L$, respectively. Let $\varphi\colon\N\to D$ and $\psi\colon\N\to E$ be bijections. By the Stone Extension Property of $\beta\N$, there are continuous maps $\Phi\colon\beta\N\to K$ and $\Psi\colon\beta\N\to L$ such that $\Phi\rstr\N=\varphi$ and $\Psi\rstr\N=\psi$. For each $n\io$ define a measure $\rho_n$ on $K\times L$ as follows:
\[\rho_n=\sum_{(x,y)\in\supp(\nu_n)}\nu_n\big(\{(x,y)\}\big)\cdot\delta_{(\varphi(x),\psi(y))},\]
it follows that $\big\|\rho_n\big\|=1$ and $\supp\big(\rho_n\big)$ is finite. Since $\seqn{\nu_n}$ weak* converges to $0$, for every $f\in C(K\times L)$ we have:
\[\lim_{n\to\infty}\int_{K\times L}fd\rho_n=\lim_{n\to\infty}\int_{\beta\N\times\beta\N}f(\Phi,\Psi)d\nu_n=0,\]
where $f(\Phi,\Psi)(x,y)=f(\Phi(x),\Psi(y))\in C(\beta\N\times\beta\N)$, so $\seqn{\rho_n}$ is also weak* convergent to $0$.
\end{proof}




Theorem \ref{theorem:main} and the aforementioned characterization from \cite{BKS19} of those spaces $C_p(X)$ which contain a complemented copy of the space $(c_0)_p$ imply immediately Corollary \ref{cor:main}. It appears that using this corollary and the Closed Graph Theorem one can easily get Cembranos' and Freniche's Theorem \ref{theorem:cemfre}, as shown in the next proposition.

\begin{proposition}\label{prop:c0p_c0}
Let $X$ be a compact space such that $C_p(X)$ contains a
complemented closed linear subspace $E$ isomorphic to $(c_0)_p$. Then, $E$ with the norm topology of $C(X)$ is complemented in $C(X)$ and isomorphic to
the Banach space $c_0$.
\end{proposition}
\begin{proof}
Let $F$ be a closed linear subspace of $C_p(X)$ such that $C_p(X)=E\oplus
F$. Then,  since the norm topology of $C(X)$ is finer than the
product topology of $C_p(X)$, the spaces $(E,\|\cdot\|)$ and
$(F,\|\cdot\|)$ (i.e. endowed with the inherited norm topology of
$C(X)$) are still closed in $C(X)$ and hence $C(X)=E\oplus F$. It is
enough now to show that $(E,\|\cdot\|)$ is isomorphic to the Banach
space $c_0$. 

Since $(E,\tau_p)$ (i.e. with the inherited product
topology of $C_p(X)$) is isomorphic to $(c_0)_p$, there is a
topology $\tau$ on $E$ stronger than $\tau_p$ and such that $(E,\tau)$ is isomorphic to $c_0$. The identity operator $T\colon(E,\|\cdot\|)\to(E,\tau)$ has closed graph, so it is continuous, and hence $\tau$ is a Banach
space topology on $E$ smaller than the norm topology of $E$. On the
other hand, the identity operator $S\colon(E,\tau)\to(E,\|\cdot\|)$
has closed graph, too, so it is also continuous, and hence the
topology $\tau$ on $E$ is greater than the norm topology of $E$. It
follows that both topologies are equal, and hence
$(E,\|\cdot\|)$ is isomorphic to the Banach space $c_0$.
\end{proof}




\end{document}